\documentclass[11pt,reqno]{amsart}

\usepackage{amscd,amssymb,amsmath,amsthm}
\usepackage[arrow,matrix]{xy}
\usepackage{graphicx}
\usepackage{cite}
\newtheorem{thm}{Theorem}

\newtheorem{lemma}{Lemma}
\newtheorem{pro}{Proposition}

\numberwithin{equation}{section} 

\newcommand{\bea}{\begin{eqnarray}}
\newcommand{\eea}{\end{eqnarray}}

\begin{document}
\title[Zebra-percolation on trees]{Zebra-percolation on Cayley trees}

\author{D. Gandolfo,  U. A. Rozikov, J. Ruiz}

 \address{D.\ Gandolfo and J.Ruiz\\Centre de Physique Th\'eorique, UMR 6207,Universit\'es Aix-Marseille
 et Sud Toulon-Var, Luminy Case 907, 13288 Marseille, France.}
\email {gandolfo@cpt.univ-mrs.fr\ \ ruiz@cpt.univ-mrs.fr}

 \address{U.\ A.\ Rozikov\\ Institute of mathematics,
29, Do'rmon Yo'li str., 100125, Tashkent, Uzbekistan.}
\email {rozikovu@yandex.ru}

\begin{abstract} We consider Bernoulli (bond) percolation with parameter $p$ on the Cayley tree of order $k$.
We introduce the notion of zebra-percolation that is
percolation by paths of alternating open and closed edges.
In contrast with standard percolation with critical threshold
at $p_c= 1/k$, we show that zebra-percolation occurs between
two critical values  $p_{{\rm c},1}$ and $p_{{\rm c},2}$
(explicitly given).
We  provide the specific formula of zebra-percolation function.

\end{abstract}
\maketitle



\section{Introduction and definitions}
Percolation on trees still remains the subject of many open
problems.
The purpose of this paper is to study  the  percolation phenomenon  by paths  of alternating open and
closed bonds. Such paths are called zebra-paths.

We consider the  Cayley tree
$\Gamma^k=(V, L)$  where each vertex has $k + 1$ neighbors with $V$
being the set of vertices and $L$ the set of bonds.
 Bonds are independently open  with probability $p$ (and closed with probability $1-p$).
We let $P_p$ be corresponding probability measure.

On this tree we fix a given vertex $e$ (the root) and consider
the following event 
\begin{equation}
E=\{\mbox{An infinite zebra-path contains the root} \}.
\end{equation}
By path we mean a collection of consecutive bonds (appearing only once) sharing a common endpoint.
  The  zebra-percolation function is defined by
 \begin{equation}
 \zeta_k(p)=P_p(E).
 \end{equation}

\medskip
The paper is organized as follows.
In Section 2 we show that  zebra-percolation occurs in the range  $p\in (p_{{\rm c},1},p_{{\rm c},2})$.  This holds as soon as $k\geq 3$ and the two critical values are explicitly given. Section 3 is devoted to standard percolation. In Section 4 we give a relation between standard percolation and zebra-percolation. The last section is devoted to some discussions and open problems.

\section{Two critical values}

The existence of two critical values is a consequence of the following dichotomy
\begin{thm}\label{t1} The zebra percolation function satisfies

\begin{itemize}
\item[1)] If 
$k^2p(1-p) < 1$, then $\zeta_k(p) = 0$.

\item[2)] If $k^2p(1-p) > 1$, then $\zeta_k(p)>0$.
\end{itemize}
\end{thm}
\begin{proof} 1)
Consider on the tree $\Gamma^k$ all paths of length $n$ starting from the root.
We will denote hereafter by  $W_n$  the set of endpoints of  these paths (excluding the root).
Let   $\mathcal F_n$    be the event that there is a zebra-path of length $n$. 
The probability  $\mathcal P_n$ for such an event is
\begin{equation}\label{p}
\mathcal P_n=\left\{\begin{array}{ll}
2(p(1-p))^{n/2},  \ \ \mbox{if} \ \ n \ \ \mbox{is even}\\[2mm]
(p(1-p))^{(n-1)/2},  \ \ \mbox{if} \ \ n \ \ \mbox{is odd}.
\end{array}
\right.
\end{equation}
 The number of  paths is at most $|W_{n}|=(k+1)k^{n-1}$.
This implies that
$$P_p(\mathcal F_n)\leq 2(k+1)k^{n-1}(p(1-p))^{[n/2]},$$ 
which, under the condition  $k^2p(1-p) < 1$, goes to 0 as $n\to \infty$. Hereafter $[\cdot]$ denotes the integer part. We then get $\zeta_k(p) = 0$.\\

2)  We shall show that if $k^2p(1-p)> 1$,
then the root zebra-percolates with positive probability. Let $X_n$ denote the number of vertices belonging to  $W_{n}$ and zebra-connected to the root. We will apply the method of second moment to the random variable $X_n$ (see, e.g.\cite{S}). We have
\begin{equation}\label{ss}
P(X_n > 0)\geq {E[X_n]^2\over E[X_n^2]}.
\end{equation}

By linearity, we have that
$E(X_{n}) = |W_{n}|\mathcal P_n$. If we can show that for some constant $M$ and for all $n$,
\begin{equation}\label{s51}
E(X^2_n )\leq M E(X_n)^2,
\end{equation}
we would then have that $P_p(X_n > 0)\geq {1\over M}$
for all $n$. The events $\{X_n > 0\}$ are decreasing and so countable additivity yields
$P_p(X_n > 0, \ \ \forall n)\geq {1\over M}$. But the latter event is the same as the event that the root
is percolating and one is done.
We now bound the second moment in order to establish (\ref{s51}). Letting $U_{v,w}$
be the event that both $v$ and $w$ are zebra-connected to the root, we have that
\begin{equation}\label{so}
E(X^2_n) = \sum_{v, w\in W_n}P_p(U_{v,w}).
\end{equation}
Now $P_p(U_{v,w}) = \mathcal P_n^2\mathcal P^{-1}_{m_{v,w}}$, where $m_{v,w}$ is the
level at which paths from $e$ to $v$ and to $w$ split. For a given $v$ and $m$, the number of $w$ with $m_{v,w}$ being
$m$ is at most $|W_n|/|W_m|$. Hence
\begin{equation}\label{so1}
E(X^2_n) \leq |W_n|\sum_{m=0}^n\mathcal P_n^2\mathcal P^{-1}_{m}|W_n|/|W_m|=E(X_n)^2\sum_{m=0}^n{1\over \mathcal P_{m}|W_m|}.
\end{equation}
If $\sum_{m=0}^\infty{1\over \mathcal P_{m}|W_m|}<\infty$,
then we would have (\ref{s51}). If $k^2p(1-p) >1$, then using formula (\ref{p}) one can see that ${1\over \mathcal P_{m}|W_m|}$ decays exponentially like $(k^2p(1-p))^{-m/2}$  giving the desired convergence.
\end{proof}

This theorem gives two critical values for the zebra-percolation which are solutions to $k^2p(1-p)=1$:
$$p_{{\rm c},1}(k)={k-\sqrt{k^2-4}\over 2k}, \ \  p_{{\rm c},2}(k)={k+\sqrt{k^2-4}\over 2k}.$$

Note that if $k\geq 3$, $0< p_{{\rm c},1}(k)<{1\over k}<{1\over 2}<p_{{\rm c},2}(k)<1$.
Moreover $p_{{\rm c},1}(k)+p_{{\rm c},2}(k)=1$. This tells that $p_{{\rm c},1}$ and $p_{{\rm c},2}$ are symmetric with respect to $1/2$.
When $k=2$,  $p_{{\rm c},1}(k)=p_{{\rm c},2}(k)=1/2$ so that no zebra-percolation occurs.

\section{On percolation function} 
Consider standard percolation model on a Cayley tree. Denote by $\theta_k(p)$ the standard percolation function, that is the probability with respect to $P_p$ that there exists an infinite cluster of open edges containing the root. We refer the reader to \cite{G}, \cite{H}, \cite{L}, \cite{P}.

\begin{pro} The function $\theta_k(p)$ satisfies
$$\theta_k(p)=\left\{\begin{array}{ll}
0, \ \ \mbox{if} \ \ \ \ p\leq {1\over k}\\[2mm]
\hat\theta_k(p),\ \ \mbox{if} \ \ p>{1\over k},
\end{array}
\right.
$$ where $\hat \theta_k(p)$ is a unique solution to the following functional equation
\begin{equation}\label{te}
\hat \theta_k(p)=1-\left(1-p\hat \theta_k(p)\right)^k, \ \ p>{1\over k}.
\end{equation}
\end{pro}
\begin{proof} Let $e$ be the root of the Cayley tree, and $S(e)$ the set of direct successors of the root.
Denote by $\mathcal A_i$ the event that vertex $i\in S(e)$ is in an infinite component, which is not connected to $e$.
Then by self-similarity we get
$$P_p(\mathcal A_i)=\theta_k(p), \ \ \mbox{for any} \ \ i\in S(e).$$
Let $\mathcal B_i$ be the event that the edge $\langle e, i\rangle$ is open and $\mathcal A_i$ holds.
Then
$$P_p(\mathcal B_i)=p\, \theta_k(p), \ \ \mbox{for any} \ \ i\in S(e).$$
Since $\mathcal B_1$,  $\mathcal B_2$, $\dots$, $\mathcal B_k$ are independent, using inclusion-exclusion principle, we get
$$\theta_k(p)=P_p\left(\bigcup_{i=1}^k\mathcal B_i\right)=$$
$$\sum_{i=1}^kP_p(\mathcal B_i)-\sum_{{i,j:\atop i<j}}P_p(\mathcal B_i\cap \mathcal B_j)+\sum_{{i,j,q:\atop i<j<q}}P_p(\mathcal B_i\cap \mathcal B_j\cap \mathcal B_q)-\dots+(-1)^{k-1}P_p\left(\bigcap_{i=1}^k\mathcal B_i\right)=$$
$$kp\theta_k(p)-{k \choose 2}(p\theta_k(p))^2+{k \choose 3}(p\theta_k(p))^3-\dots +(-1)^{k-1}(p\theta_k(p))^k=$$
$$1-\left(1-p\theta_k(p)\right)^k.$$
Hence $\theta_k(p)$ is a fixed point of the function
$$f(x)=1-(1-px)^k, \ \ x\in [0,1].$$
The proof is then completed by using the following lemma:
\end{proof}

\begin{lemma} The function $f$ satisfies
\label{l}
\begin{itemize}
\item[i.] If $p\leq {1\over k}$ then the function $f(x)$ has a unique fixed point $0$.
\item[ii.] If $p>{1\over k}$ then the function $f(x)$ has two fixed points $0$ and $\hat \theta$.
\end{itemize}
\end{lemma}
\begin{proof} Note that $0$ is a fixed point of $f$. On the other hand, $f(1)=1-(1-p)^k$ and
$$ f'(x)=kp(1-px)^{k-1}\geq 0,\ \ f''(x)=-k(k-1)p^2(1-px)^{k-2}\leq 0, \ \  x\in [0,1].$$
Hence $f$ is increasing and concave. It is easy to see that $f$ has a unique fixed point $\hat \theta\in (0,1]$ when $f'(0)=kp>1$ and  no fixed point when $f'(0)=kp\leq 1$.
This completes the proof.\end{proof}

Simple computations show that
 $$\theta_2(p)=\left\{\begin{array}{ll}
0, \ \ \mbox{if} \ \ p\leq {1\over 2}\\[2mm]
{2p-1\over p^2},\ \ \mbox{if} \ \ p>{1\over 2}.
\end{array}
\right.$$
and

 $$\theta_3(p)=\left\{\begin{array}{ll}
0, \ \ \mbox{if} \ \ p\leq {1\over 3}\\[2mm]
{2(3p-1)\over p(3p+\sqrt{p(4-3p)})},\ \ \mbox{if} \ \ p>{1\over 3}.
\end{array}
\right.$$

The general solution is given through the inverse function
\begin{pro}\label{l2} The function $\hat \theta_k(p)$, $p>1/k$, $k\geq 2$ is
invertible with inverse
\begin{equation}
\hat\theta_k^{-1}(p)={1-\sqrt[k]{1-p}\over p}.
\end{equation}
\end{pro}
\begin{proof} First we shall prove that $ \hat\theta_k(p)$ is one-to-one.
 For $p_1,p_2\in (1/k,1)$, we get from equation (\ref{te}) 
 \begin{equation}\label{11}
 \hat\theta_k(p_1)-\hat\theta_k(p_2)=\left[(p_1-p_2)\hat\theta_k(p_1)+
 p_2\left(\hat\theta_k(p_1)-\hat\theta_k(p_2)\right)\right]\cdot \mathcal U,
 \end{equation}
 where $\mathcal U=\sum_{i=0}^{k-1}(1-p_1\hat\theta_k(p_1))^{k-1-i}(1-p_2\hat\theta_k(p_2))^i>0$. 

 Since $ \hat\theta_k(p)>0$ for any $p>1/k$, if $\hat\theta_k(p_1)=\hat\theta_k(p_2)$ then
 from equality (\ref{11}) we get $p_1=p_2$. Hence $\hat\theta_k(p)$ is one-to-one, i.e. invertible.

Solving  the equation $x=1-(1-px)^k$ with respect to $p$ for $x\in [0,1]$, we get
$p=g(x)=x^{-1}(1-\sqrt[k]{1-x})$. Now by (\ref{te}) we have $p=g(\hat\theta_k(p))$ for any $p>{1\over k}$. Hence $g$ is the inverse function of $\hat\theta_k(p)$.
\end{proof}

Note that the function $\theta_k(p)$ has following properties:
\begin{enumerate}
\item
$\theta_k(p)$ is nondecreasing in $p$
\item
$\theta_k(1/k)=0$, $\theta_k(1)=1$, $\theta_k(p)\ne 1$ for any $p<1$
\item
$\theta_k(p)$ is differentiable for any $p\ne {1/k}$.
\end{enumerate}

\section{Relation between standard and zebra percolation}

Starting from the Cayley tree $\Gamma^k=(V,L)$, we construct a new tree $\hat\Gamma^k=(\hat V, \hat L)$ as follows (see Fig. 1)
$$\hat V=\bigcup_{m=0}^\infty W_{2m},\ \ \hat L=\bigcup_{m=0}^\infty\{(x,z): x\in W_{2m}, z\in S(y), y\in S(x)\},$$
where $S(x)$ denotes the set of direct successors of $x$.

It is easy to see that $\hat\Gamma^k$ is a regular tree of order $k^2$ (except on the root).

We denote by $l$ an edge in $L$ and by $\lambda$ and edge in $\hat L$. Note that any edge $\lambda\in \hat L$ can be represented by two edges $l_1, l_2\in L$, which have a common endpoint. We write this as $\lambda=(l_1,l_2)$,
moreover $l_1$ is the closer to the root of the Cayley tree.

Now for a given configuration $\sigma\in \Omega=\{0,1\}^L$ we define a configuration $\phi\in \Phi=\{-1,0,+1\}^{\hat L}$ as the following (see Fig. 1)
$$\phi(\lambda)=\phi_\sigma(\lambda)=\left\{\begin{array}{lll}
-1, \ \ \mbox{if} \ \ \sigma(l_1)=0, \, \sigma(l_2)=1\\[2mm]
0, \ \ \ \ \mbox{if} \ \ \sigma(l_1)=\sigma(l_2)\\[2mm]
1, \ \ \ \ \mbox{if} \ \ \sigma(l_1)=1, \, \sigma(l_2)=0.
\end{array}\right.$$
\begin{center}
\includegraphics[width=14cm]{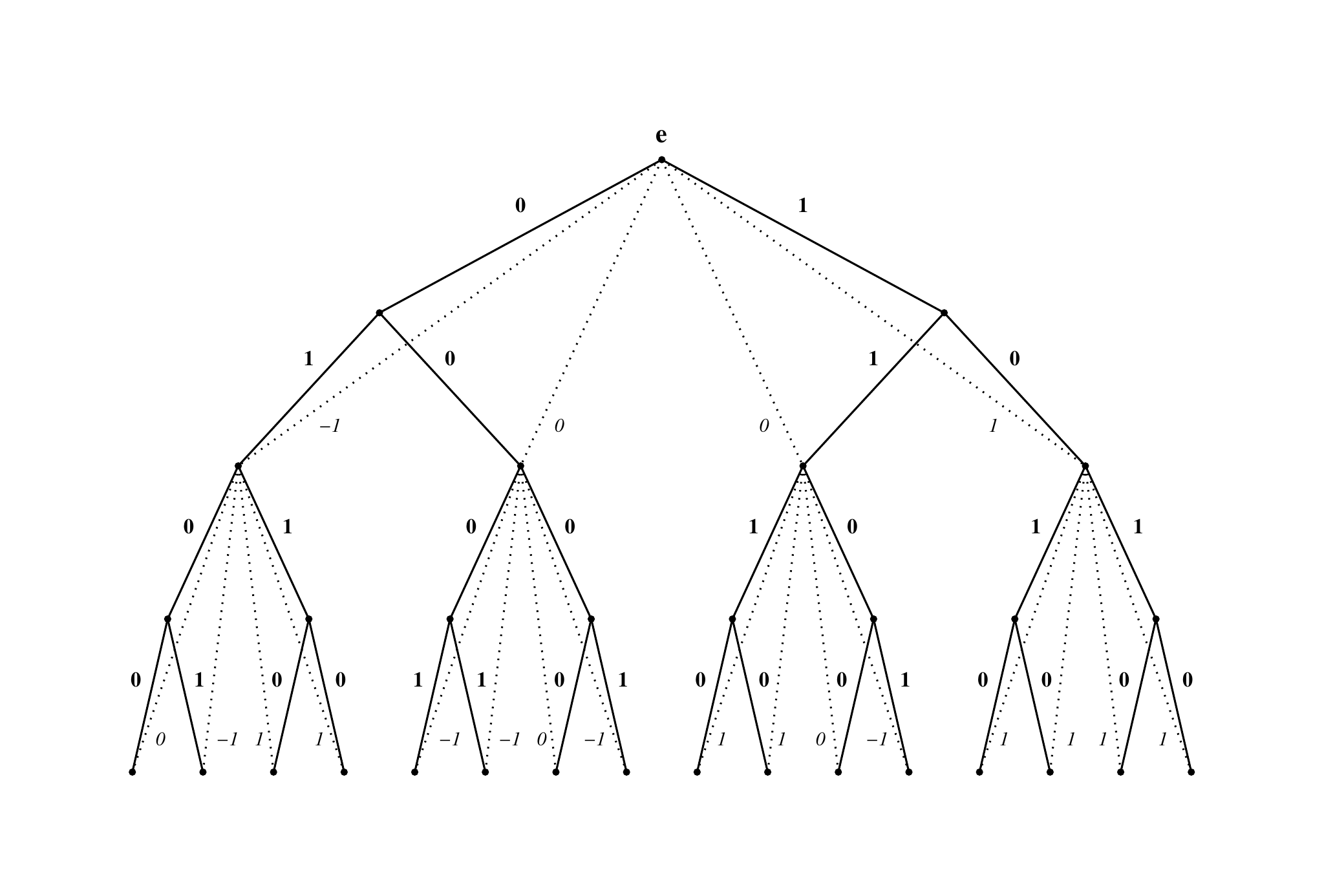}

{\footnotesize \noindent Fig.~1.
Correspondence between configurations $\sigma$ on $\Gamma^2$ (solid lines) and $\phi$ on $\hat{\Gamma}^2$ (dotted lines).}
\end{center}

A given configuration $\phi$ divides the set $\hat L$ into clusters of $(+)$ and $(-)$ bonds. \\

We speak of the edge $\lambda\in \hat L$ as being open with probability $q$ (in $\phi$) if
$\phi(\lambda)\ne 0$ and as being closed if $\phi(\lambda)=0$.
Let $\mu_q$ be corresponding product measure.
Denote
\begin{equation}\label{t}
\theta_{k^2}(q)=\mu_q(|\hat C|=\infty).
\end{equation}

By our construction the following is obvious

\begin{pro}\label{p2} The functions $\zeta_k(p)$ and $\theta_k(p)$  are related by
\begin{equation}\label{zt}
\zeta_k(p)=\theta_{k^2}(p(1-p)).
\end{equation}
\end{pro}

This proposition provides an alternative proof of Theorem \ref{t1}.
By properties of $\theta_{k^2}(p)$ we get $\zeta_k(p)=0$ iff $p(1-p)\leq 1/k^2$ and
$\zeta_k(p)>0$ iff $p(1-p)>1/k^2$. 

The two critical values $p_{{\rm c},1}$ and $p_{{\rm c},2}$ are the solutions of $p(1-p)=1/k^2$.

By Proposition \ref{p2} we get
\begin{thm}\label{t3} The function $\zeta_k(p)$ has the  following properties:
\begin{itemize}
\item[a.] $\zeta_k(p)$ is increasing in $p\in [0,1/2]$, and deacreasing in  $p\in [1/2, 1]$.

\item[b.] $\zeta_k(p_{{\rm c},1})=\zeta_k(p_{{\rm c},2})=0$,  $\max_p\zeta_k(p)=\zeta_k(1/2)=\theta_{k^2}(1/4).$

\item[b.]  $\zeta_k(p)$ is differentiable on $[0,1]\setminus \{p_{{\rm c}, 1}, p_{{\rm c}, 2}\}$.

\item[c.] there is no zebra-percolation for $k=2$.

\end{itemize}
\end{thm}
The graphs of functions $\theta_k(p)$, $\theta_{k^2}(p)$ and $\zeta_k(p)$ are presented for $k=3$ in Fig.2.

\begin{center}
\includegraphics[width=10cm]{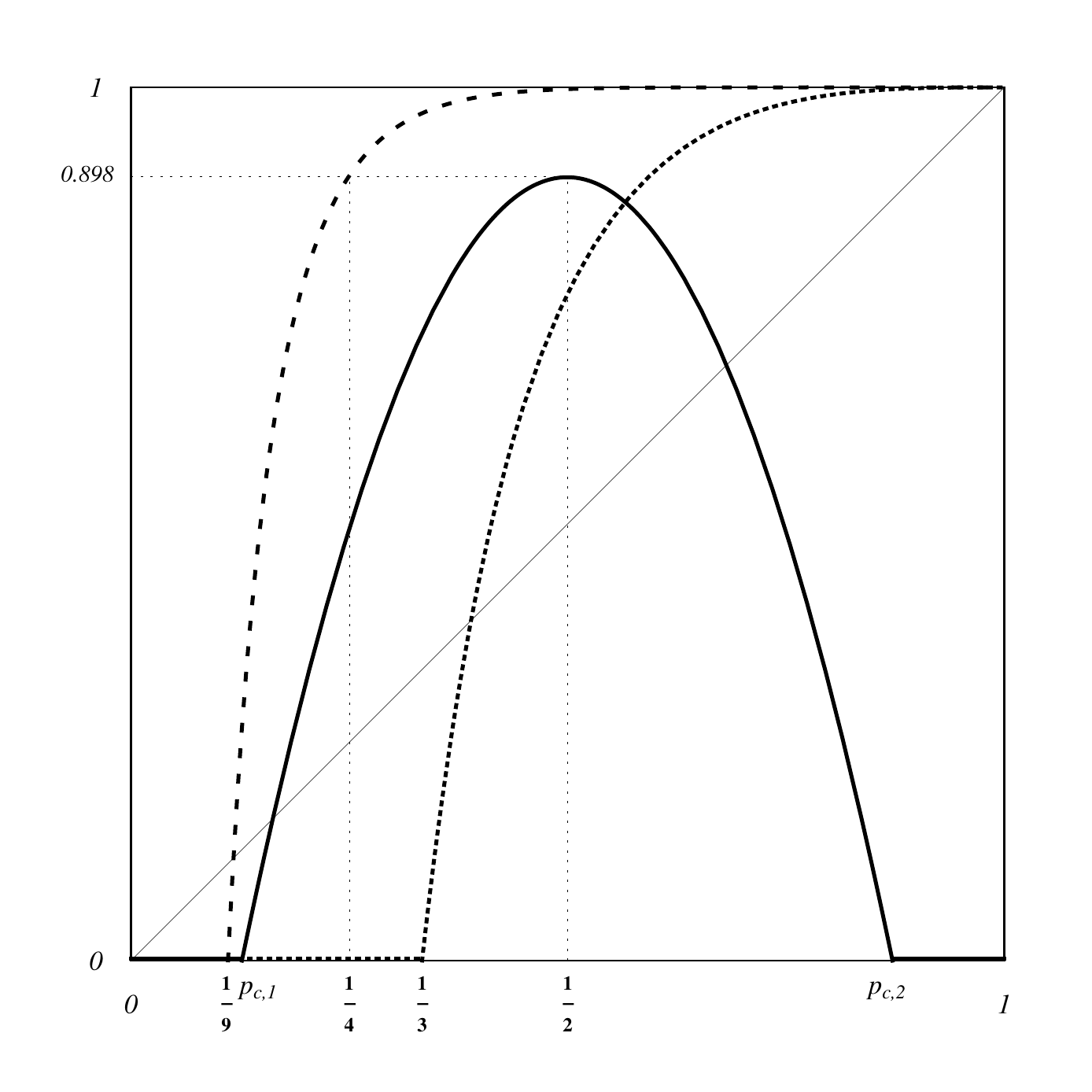}

{\footnotesize \noindent Fig.~2.
Graphs of $\theta_3(p)$ (dashed line), $\theta_9(p)$  (dotted line), and $\zeta_3(p)$  (solid line).}
\end{center}

\section{Open problems}

An interesting problem in percolation theory is to study the distribution of the
number of vertices in clusters and geometric properties
of open clusters when $p$ is close to the critical value $p_c$. It
is believed that some of these properties are
universal, i.e., depend only on the dimension
of the graph. Some open problems are in order.\\

{\bf Problem 1.} {\sl Study distribution of the
number of vertices and geometric properties of the zebra-connected clusters (made of zebra paths) when $p$ is close to $p_{{\rm c},1}$ or $p_{{\rm c},2}$}.\\

It is known that $\mathbb Z^d$ for large $d$ behaves in many respects like a
regular tree. \\

{\bf Problem 2.} {\sl Define a notion of zebra-connected component on $\mathbb Z^d$.  Find the critical value(s) for zebra-percolation on} $\mathbb Z^d$.\\

When an infinite cluster exists, it is natural to ask how many there are (see e.g. \cite{Be}). \\

{\bf Problem 3.}  {\sl How many infinite cluster exist for zebra-percolation} ?

\bigskip

\section*{ Acknowledgements}

 U.\,Rozikov thanks CNRS for financial support and  the Centre de Physique Th\'eorique - Marseille for kind hospitality during his visit (September-December 2012).

\end{document}